\documentclass[reqno]{amsart}
\usepackage[mathscr]{eucal}
\usepackage{amsmath,amsthm,amssymb}
\usepackage[hidelinks]{hyperref}
\usepackage[alphabetic]{amsrefs}
\usepackage{amsfonts}
\usepackage{dsfont}
\usepackage{latexsym}
\usepackage{mathrsfs}
\usepackage{layout}
\usepackage{bm}
\usepackage{verbatim}
\usepackage{titlecaps, cleveref}
\usepackage{esint} 
\usepackage[british]{babel}
\usepackage{csquotes}
\usepackage{microtype} 
\usepackage{enumitem} 
\usepackage{thmtools}

\newtheorem{theorem}{Theorem}[section]
\newtheorem{lemma}[theorem]{Lemma}

\newtheorem{corollary}[theorem]{Corollary}

\newtheorem{remark}[theorem]{Remark}
\newtheorem{definition}[theorem]{Definition}

\numberwithin{equation}{section}

\newcommand{\Z}{\ensuremath{\mathbb{Z}}}
\newcommand{\R}{\ensuremath{\mathbb{R}}}
\renewcommand{\H}{\mathscr{H}} 

\newcommand{\D}{\ensuremath{\mathbb{D}}}

\newcommand{\dx}{\ensuremath{\, \mathrm{d} x}}
\newcommand{\dX}{\ensuremath{\, \mathrm{d} X}}
\newcommand{\dY}{\ensuremath{\, \mathrm{d} Y}}
\newcommand{\dZ}{\ensuremath{\, \mathrm{d} Z}}
\newcommand{\dt}{\ensuremath{\, \mathrm{d} t}}
\newcommand{\ds}{\ensuremath{\, \mathrm{d} s}}
\newcommand{\dw}{\ensuremath{\, \mathrm{d} \omega}}
\newcommand{\dsig}{\ensuremath{\, \mathrm{d} \sigma}}
\newcommand{\dtau}{\ensuremath{\, \mathrm{d} \tau}}
\newcommand{\dH}{\ensuremath{\, \mathrm{d} \H}}

\newlist{steps}{enumerate}{2}
\setlist[steps]{wide} 
\setlist[steps,1]{label={\textbf{Step \arabic*:}}} 
\setlist[steps,2]{label={\textbf{Step \arabic{stepsi}.\alph*:}}} 

\renewcommand{\epsilon}{\varepsilon}
\newcommand{\st}{\ensuremath{:}}
\newcommand{\loc}{\mathrm{loc}}

\DeclareMathOperator*{\supp}{supp}
\DeclareMathOperator{\BMO}{BMO}
\DeclareMathOperator{\Lip}{Lip}
\DeclareMathOperator{\dist}{dist}
\DeclareMathOperator{\diam}{diam}
\DeclareMathOperator*{\di}{div}

\crefname{corollary}{Corollary}{Corollaries}
\crefname{theorem}{Theorem}{Theorems}
\crefname{lemma}{Lemma}{Lemmas}
\crefname{proposition}{Proposition}{Propositions}
\crefname{equation}{}{} 

\begin{document}
\title{Parabolic Regularity and Dirichlet boundary value problems}

\author{Martin Dindo\v{s}}
\address{School of Mathematics \\
	The University of Edinburgh and Maxwell Institute of Mathematical Sciences, UK}
\email{m.dindos@ed.ac.uk}

\author{Luke Dyer}
\address{School of Mathematics \\
	The University of Edinburgh and Maxwell Institute of Mathematical Sciences, UK}
\email{l.dyer@sms.ed.ac.uk}
\date{}
\begin{abstract}
	We study the relationship between the Regularity and Dirichlet boundary value problems for parabolic equations of the form $Lu=\text{div}(A \nabla u)-u_t=0$ in $\Lip(1,1/2)$ time-varying cylinders, where the coefficient matrix $A = \left[ a_{ij}(X,t)\right] $ is uniformly elliptic and bounded.

	We show that if the Regularity problem $(R)_p$ for the equation $Lu=0$ is solvable for some $1<p<\infty$ then the Dirichlet problem $(D^*)_{p'}$ for the adjoint equation $L^*v=0$ is also solvable, where $p'=p/(p-1)$.
	This result is an analogue of the result established in the elliptic case by Kenig and Pipher \cite{KP93}.
	In the parabolic settings in the special case of the heat equation in slightly smoother domains this has been established by Hofmann and Lewis \cite{HL96} and Nystr\"om \cite{Nys06} for scalar parabolic systems. In comparison, our result is abstract with no assumption on the coefficients beyond the ellipticity condition and is valid in more general class of domains.
\end{abstract}
\maketitle

\section{Introduction}
We are interested in the relationship between the solvability of the Regularity and the Dirichlet boundary value problems for parabolic operators
\[
	L= \di(A \nabla \cdot)-\partial_t
\]
on $\Lip(1,1/2)$ cylinders $\Omega$.
These domains are bounded and Lipschitz in spatial variables, unbounded and $\Lip_{1/2}$ in time.
Furthermore, we assume that the matrix $A(X,t)$ satisfies an ellipticity condition, and its coefficients are bounded and measurable.

	The question of solvability of various boundary value problems for parabolic PDEs on time-varying domains has long history. Recall, that in the elliptic settings \cite{Dah77} has shown that, in a Lipschitz domain, the harmonic measure and surface measure are mutually absolutely continuous, and that the elliptic Dirichlet problem is solvable with data in $L^2$ with respect to surface measure.
	R. Hunt then asked whether Dalhberg's result held for the heat equation in	domains whose boundaries are given locally as functions $\phi(x,t)$, Lipschitz in the spatial variable.
	It was conjectured (due to the natural parabolic scaling) that the correct regularity of $\phi(x,t) $ in the time variable $t$ should be a H\"older condition of order $1/2$ in $t$ (denoted $\Lip_{1/2}$ in $t$). It turns out that under this assumption the parabolic measure associated with the equation \eqref{E:pde} is doubling \cite{Nys97}.
	
	This is the class of domains we work on. It is worth pointing out however that in order to answer R. Hunt's question positively one has to consider more regular domains. This follows from the counterexample of \cite{KW88} where it was shown that under just the $\Lip(1,1/2)$ condition on the domain $\Omega$ the associated caloric measure (that is the measure associated with the operator $\partial_t-\Delta$) might not be mutually absolutely continuous with the natural surface measure.
	The issue was resolved in \cite{LM95} where it was established that mutual absolute continuity of caloric measure and a certain parabolic analogue of surface measure holds when $\phi$ has $1/2$ of a time derivative in the parabolic $\BMO (\R^n)$ space, which is a slightly stronger condition than $\Lip_{1/2}$.
	\cite{HL96} subsequently showed that this condition was sharp.
	In particular in this paper the authors has solved the $L^2$ Dirichlet problem for the heat equation in graph domains of Lewis-Murray type. A related  class of localised domains in which parabolic boundary value problems are solvable was considered in \cite{Riv14} as well as in \cites{DH16, DPP16}. The paper \cite{DH16} has established $L^p$ solvability for parabolic Dirichlet problem under assumption that the coefficients satisfy certain natural small Carleson condition which also appears for elliptic PDEs. The second paper \cite{DPP16} finds sufficient and necessary condition for the parabolic measure to be $A_\infty$ with respect to the parabolic analogue of the surface measure.\vglue1mm
			
The study of the heat equation in non-smooth domains, or more generally of parabolic operators with non-smooth coefficients, has historically followed the development of the elliptic theory with some delay due to new challenges presented by the parabolic term.

	Our result is also motivated by a result proven in the elliptic setting by \cite{KP93} where, amongst other relationships, they show that $(R_p)$ implied $(D^*)_{p'}$ for elliptic operators $\di(A \nabla \cdot)$ in bounded Lipschitz domains. This has been observed previously for some specific parabolic PDEs (such as the heat equation and constant coefficient systems \citelist{\cite{HL96}*{p.~418} \cite{Nys06}} respectively). \cite{Nys06} also shows that no duality can be expected between Dirichlet and Neumann boundary value problems in non-smooth time-varying domains.
	
	In our result we remove any restrictions on the coefficients of the scalar elliptic operator (beyond the ellipticity hypothesis) and establish the result on the largest reasonable class of domains. It is worth pointing out that due to the roughness of the coefficients and of the boundary of these domains the usual techniques (such as layer potentials and Fourier methods) are not available.
			
	Our main result proves that if the Regularity problem $(R_p)$ for the operator $L$ on the domain $\Omega$ is solvable for some $1<p<\infty$ ($(R_p)$ has boundary data in a Sobolev space $L^p_{1,1/2}(\partial\Omega)$,  which is a space of functions with spatial derivatives and a half-time derivative in $L^p$) then 
	the Dirichlet problem $(D^*)_{p'}$ ($(D^*_{p'})$ has boundary data in $L^{p'}(\partial\Omega)$) for the adjoint operator 
	$$L^*= \di(A^* \nabla \cdot)+\partial_t$$ is also solvable on the domain $\Omega$.
	
	Observe that $L^*$ is a backward in time parabolic operator. This however does not causes any issues as by
	the change of variables of $v(X,t) = u(X,-t)$ and $\tilde{A}(X,-t) = A(X,t)$ we see that $L^*u=0$ on $\Omega$
	is equivalent to
	\begin{equation}
	\label{E:pde:adjoint:reflected}
	\tilde{L}v = \di (\tilde{A}^* \nabla v)-v_t=0 \qquad \text{on } \tilde{\Omega},
	\end{equation}
	where $\tilde{\Omega}$ is the reflection of $\Omega$ in the $t$ variable i.e.\ $\tilde{\Omega} = \{(X,-t) \st (X,t) \in \Omega\}$. Hence, the solvability of the $L^{p'}$ Dirichlet problem for the operator $L^*$ on $\Omega$
	is equivalent to the solvability of the $L^{p'}$ Dirichlet problem for the operator $\tilde{L}$ on $\tilde{\Omega}$. Here $\tilde{L}v=0$ is the usual forward in time parabolic PDE.	

\begin{theorem}
	\label{T:Rp}
	Let $\Omega$ be a $\Lip(1,1/2)$ cylinder, as in \cref{D:domain}, with character $(\ell, N,C_0)$.	Let $A(X,t)$ be bounded, measurable and elliptic, that is
	\begin{equation}
		\label{E:elliptic}
		\lambda |\xi|^2 \leq \sum_{i,j} a_{ij}(X,t) \xi_i \xi_j \leq \Lambda |\xi|^2
	\end{equation}
	for all $\xi \in \R^n$ and a.e. $X \in \R^n$, $t\in \R$.
	Let the Regularity problem $(R)_p$ be solvable for the equation
	\begin{equation}
		\label{E:pde}
		\begin{cases}
			u_t = \di (A \nabla u) & \text{in } \Omega \subset \R^{n+1}, \\
			\, u   = f             & \text{on } \partial \Omega,
		\end{cases}
	\end{equation}
	for some $1<p<\infty$. Then the Dirichlet problem $(D^*)_{p'}$ is solvable for the adjoint equation
	\begin{equation}
		\label{E:pde:adjoint}
		\begin{cases}
			-u_t = \di (A^* \nabla u) & \text{in } \Omega \subset \R^{n+1}, \\
			\quad\, u   = f           & \text{on } \partial \Omega,
		\end{cases}
	\end{equation}
	where $p'=p/(p-1)$.
\end{theorem}

The paper is organized as follows. In section 2 we introduce $\Lip(1,1/2)$ cylinders, a suitable local pullback transformation, parabolic non-tangential maximal functions, and the $L^p_{1,1/2}$ parabolic Sobolev space on $\R^n$ and domains. In section 3 we state and prove some basic results for parabolic equations and some lemmas needed for the proof of \cref{T:Rp}. In section 4 we prove our main result \cref{T:Rp}.

\textit{Acknowledgements.} Luke Dyer was supported by The Maxwell Institute Graduate School in Analysis and its Applications, a Centre for Doctoral Training funded by the UK Engineering and Physical Sciences Research Council (grant EP/L016508/01), the Scottish Funding Council, the University of Edinburgh and Heriot-Watt University.

\section{Preliminaries}

Here and throughout we consistently use $\nabla u$ to denote the gradient in the spatial variables, $u_t$ or $\partial_t u$ the gradient in the time variable and use $Du=(\nabla u, \partial_t u)$ for the full gradient of $u$.

\subsection{Parabolic measure}
It is well known by the Perron-Wiener-Brelot method \cite{Ekl79} that the parabolic PDE (\ref{E:pde}) with continuous boundary data is uniquely solvable (c.f.\ \cref{R:max}) and that there exists a unique measure $\omega^{(X,t)}$, called the \textit{parabolic measure}, such that
\begin{equation}
	\label{E:pm}
	u(X,t)=\int_{\partial \Omega} f(y,s) \dw^{(X,t)}(y,s)
\end{equation}
for all continuous data $f$.
Under the assumptions of \cref{D:domain} this measure is doubling (\cite{Nys97}).
As $\omega^{(X,t)}$ is a Borel measure, it follows that we can use \eqref{E:pm} to extend the solvability of \eqref{E:pde} to a class of bounded Borel measurable functions $f$.

\subsection{Lip(1,1/2) cylinders}

In this subsection we recall the class of $\Lip(1,1/2)$  time-varying cylinders in \cite{Nys97} whose boundaries are given locally as functions $\phi(x,t)$, Lipschitz in the spatial variable and $\Lip_{1/2}$ in the time variable.
At each time $\tau\in\R$ the set of points in $\Omega$ with fixed time $t=\tau$, that is
$\Omega_\tau=\{(X,\tau)\in\Omega\}$, will be a non-empty bounded Lipschitz domain in $\R^n$.
We start with few preliminary definitions, motivated by the standard definition of a Lipschitz domain.

\begin{definition}
	$\Z \subset \R^n\times \R$ is an \textit{$\ell$-cylinder} of diameter $d$ if there
	exists a coordinate system $(x_0,x,t)\in \R\times\R^{n-1}\times \R$ obtained from the original coordinate system only by translation in spatial and time variables and rotation in the spatial variables such that
	\[
		\Z = \{ (x_0,x,t) \st |x|\leq d, |t|\leq d^2, -(\ell +1)d \leq x_0 \leq (\ell + 1)d \}
	\]
	and for $s>0$
	\[
		s\Z:=\{(x_0,x,t) \st |x|<sd, |t|\leq s^2d^2, -(\ell +1)sd \leq x_0 \leq (\ell +1)sd \}.
	\]
\end{definition}

\begin{definition}
	\label{D:domain}
	$\Omega\subset \R^n\times \R$ is a \textit{$\Lip(1,1/2)$ cylinder} with character $(\ell,N,C_0)$ if there exists a positive scale $r_0$ such that for any time $\tau\in\R$	there are at most $N$ $\ell$-cylinders $\{{\Z}_j\}_{j=1}^N$ of diameter $d$, with $\frac{r_0}{C_0}\leq d \leq C_0 r_0$, satisfying the following:
	\begin{enumerate}
		\item $8{\Z}_j \cap \partial\Omega$ is the graph $\{x_0=\phi_j(x,t)\}$ of a
		      function $\phi_j$ such that
		      \begin{equation}
			      \label{E:L1}
			      |\phi_j(x,t)- \phi_j(y,s)| \leq \ell \left( |x-y| + |t-s|^{1/2}\right) \text{ and }
			      \phi_j(0,0)=0.
		      \end{equation}

		\item $\displaystyle \partial\Omega\cap\{|t-\tau|\le d^2\}=\bigcup_j ({\Z}_j \cap 	\partial\Omega)$.

		\item In the coordinate system $(x_0,x,t)$ of the $\ell$-cylinder $\Z_j$
		      \[
			      \displaystyle{\Z}_j \cap \Omega \supset \left\{
			      (x_0,x,t)\in\Omega \st |x|<d, |t|<d^2, \delta(x_0,x,t) := \dist \left( (x_0,x,t),\partial\Omega
			      \right) \leq \frac{d}{2}\right\}.
		      \]
		      Here and throughout $\dist$ is the parabolic distance $\dist[(X,t),(Y,\tau)]=|X-Y|+|t-\tau|^{1/2}$.
	\end{enumerate}
\end{definition}

The \textit{parabolic norm} $\|(X,t)\|$ on $\R^n \times \R$ is defined as the unique positive solution $\rho$ to the following equation
\begin{equation}
	\label{E:par-norm}
	\frac{|X|^2}{\rho^2} + \frac{t^2}{\rho^4} = 1.
\end{equation}
One can easily show that $\|(X,t)\| \sim |X| + |t|^{1/2}$ and that this norm has the correct scaling.

\begin{remark}
	It follows from this definition that for each $\tau\in\R$ the time-slice $\Omega_\tau=\Omega\cap\{t=\tau\}$ of a $Lip(1,1/2)$ cylinder $\Omega\subset \R^n\times \R$ is a bounded Lipschitz domain in $\R^n$ with character $(\ell,N,C_0)$.
	Due to this fact, the Lipschitz domains $\Omega_{\tau}$ for all $\tau\in\R$ have all uniformly bounded diameter (from below and above).
	That is
	\[
		\inf_{\tau\in\R}\diam(\Omega_\tau) \sim r_0 \sim \sup_{\tau\in\R}\diam(\Omega_\tau),
	\]
	where $r_0$ is the scale from Definition \ref{D:domain} and the implied constants in the estimate only depend on $N$ and $C_0$.
	In particular, if ${\mathcal O}\subset \R^n$ is a bounded Lipschitz domain then the parabolic cylinder $\Omega={\mathcal O}\times \R$ is an example of a domain satisfying Definition \ref{D:domain}.
\end{remark}

\begin{definition}
	\label{D:measure}
	Let $\Omega\subset \R^n\times \R$ be a $Lip(1,1/2)$ cylinder with character $(\ell, N,C_0)$.
	We define the \textit{measure} $\sigma$ on sets $A\subset \partial\Omega$ to be
	\begin{equation}
		\label{E:sigma}
		\sigma(A)=\int_{-\infty}^\infty {\H}^{n-1}\left(A\cap\{(X,t)\in\partial\Omega\}\right) \dt,
	\end{equation}
	where ${\H}^{n-1}$ is the $n-1$ dimensional Hausdorff measure on the Lipschitz boundary $\partial\Omega_\tau = \{(X,\tau)\in\partial\Omega\}$.
\end{definition}

We consider solvability of the $L^p$ Dirichlet and $L^p$ regularity boundary value problems with respect to the measure $\sigma$.
The measure $\sigma$ may not be comparable to the usual surface measure on $\partial\Omega$: in the $t$-direction the functions $\phi_j$ from the Definition \ref{D:domain} are only $\Lip_{1/2}$ and hence the standard surface measure might not be locally finite.
However, our definition assures that for any $A\subset 8\Z_j$, where $\Z_j$ is an $\ell$-cylinder, we have
\begin{equation}
	\label{E:comp}
	\H^n(A) \sim \sigma\left(\{(\phi_j(x,t),x,t):\,(x,t)\in A\}\right),
\end{equation}
where the constants in \cref{E:comp}, by which these measures are comparable, only depend on the $\ell$ of the character $(\ell,N,C_0)$ of the domain $\Omega$.
If $\Omega$ has a smoother boundary, such as Lipschitz (in all variables) or better, then the measure $\sigma$ is comparable to the usual $n$-dimensional Hausdorff measure $\H^n$.
In particular, this holds for a parabolic cylinder $\Omega={\mathcal O}\times \R$.

\begin{definition} Let $\Omega$ be a $\Lip(1,1/2)$ cylinder from Definition \ref{D:domain}.
	For $(y,s)\in\partial\Omega$, $(X,t), (Z, \tau) \in \Omega$ and $r>0$ we write:
	\begin{align*}
		B_r(X,t)      & =\{(Z,\tau)\in{\R}^{n}\times \R \st \dist[(X,t),(Z,\tau)]<r \},                                                        \\
		Q_{r} (X, t)
		              & = \{ (Z,\tau) \in \R^{n}\times\R \st |x_i - z_i| < r \text{ for all } 0 \leq i \leq n-1, \, | t - \tau |^{1/2} < r \}, \\
		\Psi_r(y,s)   & = \{ (Z,\tau) \in \R^{n}\times\R \st |x_0 - z_0| < (\ell + 1)r, \, |x_i - z_i| < r, \, | t - \tau |^{1/2} < r \},      \\
		\Delta_r(y,s) & = \partial \Omega\cap B_r(y,s),	\qquad T(\Delta_r) = \Omega\cap B_r(y,s),                                               \\
		\delta(X,t)   & =\inf_{(y,s)\in \partial\Omega} \dist[(X,t),(y,s)].
	\end{align*}
\end{definition}

\begin{definition}[Corkscrew points]
	Let $\Omega$ be a $\Lip(1,1/2)$ cylinder from \cref{D:domain} and $r_0>0$ the scale defined there.
	For any boundary ball $\Delta_r=\Delta_r(Y,s)\subset\partial\Omega$ with $0<r\lesssim r_0$ we say that a point $(X,t)\in\Omega$ is a \textit{corkscrew point} of the ball $\Delta_r$ if
	\[
		t=s+2r^2 \quad \text{and}  \quad \delta(X,t)\sim r\sim \dist[(X,t),(Y,s)].
	\]
	That is the point $(X,t)$ is an interior point of $\Omega$ of distance to the ball $\Delta_r$ and the boundary $\partial\Omega$ of order $r$.
	The point $(X,t)$ lies at the time of order $r^2$ further than the times for the ball $\Delta_r$.
	Finally, the implied constants in the definition above only depend on the domain $\Omega$ but not on $r$ and the point $(Y,s)$.

	Each ball of radius $0<r\lesssim r_0$ has infinitely many corkscrew points; for each ball we choose one and denote it by $V(\Delta_r)$ or if there is no confusion to which ball the corkscrew point belongs just $V_r$.
\end{definition}

\begin{remark}
	Given the fact that the time slices $\Omega_\tau$ of the domain $\Omega$ are of approximately diameter $r_0$ the corkscrew points do not exists for balls of sizes $r \gg r_0$.
\end{remark}

\subsection{Parabolic Non-tangential Cones and Maximal Functions}

We proceed with the definition of parabolic non-tangential cones. We define the cones in a (local) coordinate system where $\Omega=\{(x_0,x,t) \st x_0>\phi(x,t)\}$.
In particular this also applies to the upper half-space $U=\{(x_0,x,t) \st x_0>0\}$.
We note here, that a different choice of coordinates (naturally) leads to different sets of cones, but the particular choice of non-tangential cones is not important as it only changes constants in the estimates for the non-tangential maximal function defined using these cones.
However the norms defined using different sets of non-tangential cones are comparable.

For a constant $a>0$, we define the \textit{parabolic non-tangential cone} at a point $(x_0,x,t)\in\partial\Omega$ as follows
\begin{equation*}
	\label{D:cone}
	\Gamma_{a}(x_0, x, t) = \left\{(y_0, y, s)\in \Omega \st |y - x| + |s-t|^{1/2} < a(y_0 - x_0), \  x_0 < y_0 \right\}.
\end{equation*}
We occasionally truncate the cone $\Gamma$ at the height $r$
\begin{equation*}
	\label{D:coner}
	\Gamma_{a}^{r}(x_0, x, t) = \left\{(y_0, y, s)\in \Omega \st |y - x| + |s-t|^{1/2} < a(y_0 - x_0), \  x_0 < y_0 < x_0 + r  \right\}.
\end{equation*}

\begin{definition}[non-tangential	maximal function]
	For a function $u : \Omega \rightarrow \R$, the \textit{non-tangential maximal function} $N_a(u): \partial\Omega \to \R$ and its truncated version at a height $r$ are defined as
	\begin{equation}
		\label{D:NTan}
		\begin{split}
			N_{a}(u)(x_0,x,t) &= \sup_{(y_0,y,s)\in \Gamma_{a}(x_0,x,t)} \left|u(y_0 , y, s)\right|, \\
			N_{a}^{r}(u)(x_0,x,t) &= \sup_{(y_0,y,s)\in \Gamma_{a}^{r}(x_0,x,t)} \left|u(y_0 , y, s)\right|\quad\mbox{for }(x_0,x,t)\in\partial\Omega.
		\end{split}
	\end{equation}

	We also define the following $L^p$ variant of the non-tangential maximal function
	\begin{equation}
		\label{D:NTanVar}
		\widetilde{N}_p(u)(x_0,x,t) =
		\sup_{(Y,s) \in\Gamma_a^r(x_0,x,t)} \left(\fint_{B_{{\delta(Y,s)}/{2}}(Y,s)}|u(Z, \tau)|^p \dZ \dtau \right)^{\frac{1}{p}}.
	\end{equation}
\end{definition}

\subsection{Parabolic Sobolev Space on $\partial\Omega$}
\label{S:paraSobolev}

When considering the appropriate function space for our boundary data we want it to have the same homogeneity as the PDE.
As a rule of thumb one derivative in time behaves like two derivatives in space and so the correct order of our time derivative should be $1/2$ if we impose data with one derivative in spatial variables.
This problem has been studied previously in \cites{HL96,HL99,Nys06}, who have followed \cite{FJ68} in defining the homogeneous parabolic Sobolev space $\dot{L}^p_{1,1/2}$ in the following way.

\begin{definition}
	\label{D:paraSob}
	The \textit{homogeneous parabolic Sobolev space} $\dot{L}^p_{1,1/2}(\R^n)$, for $1 < p < \infty$, is defined to consist of equivalence classes of functions $f$ with distributional derivatives satisfying $\|f\|_{\dot{L}^p_{1,1/2}(\R^n)} < \infty$, where
	\begin{equation}
		\|f\|_{\dot{L}^p_{1,1/2}(\R^n)} = \|\D f\|_{L^p(\R^n)}
	\end{equation}
	and
	\begin{equation}
		(\D f)\,\widehat{\,}\,(\xi,\tau) := \|(\xi,\tau)\| \widehat{f}(\xi,\tau).
	\end{equation}

	We also define the \textit{inhomogeneous parabolic Sobolev space} $L^p_{1,1/2}(\R^n)$ as an equivalence class of functions $f$ with distributional derivatives satisfying  $\|f\|_{L^p_{1,1/2}(\R^n)} < \infty$, where
	\begin{equation}
		\|f\|_{L^p_{1,1/2}(\R^n)} = \|\D f\|_{L^p(\R^n)} + \|f\|_{L^p(\R^n)}.
	\end{equation}
\end{definition}

Other authors \cites{Bro89,Bro90,HL99,Mit01,Nys06,CRS15} have only considered either Lipschitz cylinders or graph domains and so have only needed to control the homogeneous norm. Because we are considering an infinite time-varying cylinder made from a local collection of graphs $\phi_j$ we need to have additional control over the $L^p$ norm of $f$ to control terms that arise from taking a smooth partition of unity.

In addition, following \cite{FR67}, we define a parabolic half-order time derivative by
\begin{equation}
	(\D_n f)\,\widehat{\,}\,(\xi,\tau) := \frac{\tau}{\|(\xi,\tau) \|} \widehat{f}(\xi,\tau).
\end{equation}
By parabolic singular integral theory \cites{FR66,FR67} we have that
\begin{align}
	\|\D f\|_{L^p(\R^n)} \sim \|\nabla f\|_{L^p(\R^n)} + \|\D_n f\|_{L^p(\R^n)}.
\end{align}

One result of this paper is that we have another characterisation of the spaces $\dot{L}^p_{1,1/2}(\R^n)$ and $L^p_{1,1/2}(\R^n)$ by an equivalent norm.
By applying Plancherel's theorem for $p=2$ we have
\begin{equation}
	\|\D f\|_{L^2(\R^n)} \sim \|D^t_{1/2} f\|_{L^2(\R^n)} + \|\nabla f\|_{L^2(\R^n)},
\end{equation}
where $D^t_{1/2}$ denotes the one-dimensional half fractional derivative of $f$ in the time variable.
We show in \cref{T:Dn} that this equivalence holds for all $1 < p < \infty$.

If $0 < \alpha \leq 2$, then for $g \in C_0^\infty(\R)$ the \textit{one-dimensional fractional differentiation operators} $D_\alpha$ are defined by
\begin{equation}
	(D_\alpha g)\,\widehat{\,}\,(\tau) := |\tau|^\alpha \widehat{g}(\tau).
\end{equation}
It is also well known that if $0 < \alpha < 1$ then
\begin{equation}
	D_\alpha g(s) = c\int_{\R} \frac{g(s) - g(\tau)}{|s-\tau|^{1 + \alpha}} \dtau
\end{equation}
whenever $s \in \R$.
If $h(x,t) \in C_0^\infty(\R^n)$ then by $D_\alpha^t h: \R^n \rightarrow \R$ we mean the function $D_\alpha h(x, \cdot)$ defined a.e.\ for each fixed $x \in \R^{n-1}$.

Since $\frac{|\tau|^{1/2}}{\|(\xi, \tau)\|}$ is an $L^p$ multiplier for $1 < p < \infty$ \cite{Ste70}*{Theorem 6, p.~109} we have the following bound.
\begin{lemma}
	\label{L:Dt}
	Let $f: \R^{n} \rightarrow \R$ and $1 < p < \infty$ then
	\begin{equation}
		\|D_t^{1/2} f\|_{L^p(\R^{n})} \lesssim \|\D f\|_{L^p(\R^{n})}.
	\end{equation}
\end{lemma}

\begin{theorem}
	\label{T:Dn}
	Let $f: \R^{n} \rightarrow \R$ and $1 < p < \infty$ then
	\begin{equation}
		\label{E:Dn}
		\|\D_n f\|_{L^p(\R^{n})} \lesssim \|D^t_{1/2} f\|_{L^p(\R^{n})} + \|\nabla f\|_{L^p(\R^{n})}.
	\end{equation}
	Therefore $\|f\|_{\dot{L}^p_{1,1/2}(\R^n)} = \|\D f\|_{L^p(\R^{n})} \sim \|D^t_{1/2} f\|_{L^p(\R^{n})} + \|\nabla f\|_{L^p(\R^{n})}$ for $1 < p < \infty$ and so
	\[	\|f\|_{L^p_{1,1/2}(\R^n)} \sim \|D^t_{1/2} f\|_{L^p(\R^{n})} + \|\nabla f\|_{L^p(\R^{n})} + \|f\|_{L^p(\R^{n})}. \]
\end{theorem}

The proof uses the same approach as \cite{HL96}*{Section 7} to obtain $L^p$ bounds instead of their mixed $\BMO$ and $L^\infty$ bounds.

\begin{proof}
	By approximation we may assume that $f \in C^\infty_0(\R^n)$ and also that $f(0) = 0$ by replacing $f$ by $f - f(0)$ and noting that $\D_n$ and $D^t_{1/2}$ map constants to the 0 element.
	Let
	\[
		m(\xi, \tau) = \frac{\tau}{|\tau|^{1/2} \|(\xi, \tau)\|}
	\]
	then we have
	\[
		(\D_n f)^\wedge (\xi, \tau) = \hat{f}(\xi, \tau) \left( m(\xi, \tau)|\tau|^{1/2}\right)
	\]
	for $(\xi, \tau) \in \R^n$, where $^\wedge$ denotes the Fourier transform on $\R^n$.

	This multiplier $m$ is not smooth enough to apply standard multiplier theorems so as in \cite{HL96} we use a smooth cut off function to split this multiplier into two.
	Let $\phi \in C^\infty_0(\R)$ be an even function with $\phi = 1$ on $(-3/2, -1/2), (1/2, 3/2)$, supported in $(-2, -1/4), (1/4, 2)$ and choose $\phi$ such that $|D^k \phi| \lesssim 2^k$ for $0 \leq k \leq n+4$.
	Let
	\[
		m^+(\xi, \tau) = m(\xi, \tau) \phi\left( \frac{\tau}{\|(\xi, \tau)\|^2} \right)
	\]
	and
	\[
		m^{++}(\xi, \tau) = \frac{|\tau|^{1/2} m(\xi, \tau) \|(\xi, \tau)\|}{|\xi|^2} (1- \phi)\left( \frac{\tau}{\|(\xi, \tau)\|^2} \right)
	\]
	then
	\[
		(\D_n f)^\wedge (\xi, \tau) = \hat{f}(\xi, \tau) \left( m^+(\xi, \tau)|\tau|^{1/2} + \frac{|\xi|^2}{\|(\xi, \tau)\|} m^{++}(\xi, \tau) \right).
	\]
	Let $m^{++}_j(\xi, \tau) = \frac{\xi_j}{\|(\xi, \tau)\|} m^{++}(\xi, \tau)$ for $0 \leq j \leq n-1$ then we show there exists singular integral operators $T_{m_j^{++}}$ and $T_{m^+}$ corresponding to $m^{++}_j$ and $m^+$ respectively such that
	\begin{equation}
		\D_n f = cT_{m^+}( D^t_{1/2} f) + c \sum_{j = 0}^{n-1} T_{m_j^{++}}( \partial_{x_j} f).
	\end{equation}
	All we have to show is that $T_{m^+}$ and $T_{m_j^{++}}$ exist and map $L^p$ into $L^p$ for $1 < p < \infty$.

	First we consider $m^+$, which is infinitely differentiable away from the origin.
	It is not hard to show that if $\gamma$ is a multi-index and $a$ a non-negative integer then
	\begin{equation}
		\label{E:m+}
		|\partial^\gamma_\xi \partial^a_\tau m^+(\xi, \tau)| \lesssim \|(\xi, \tau)\|^{-(|\gamma| + 2a)},
	\end{equation}
	for $1 \leq a + |\gamma| \leq n + 4$, and that $|m^+(\xi, \tau)| \lesssim 1$.
	By singular integral with mixed homogeneity theory \cite{FR66}*{p.~28} we have that $T_{m^+}$ exists and is bounded on $L^p$ for $1 < p < \infty$.

	Similarly considering $m_j^{++}$, by \cite{HL96}*{(7.10)-(7.11)} we have
	\begin{equation}
		\label{E:m++}
		|\partial^\gamma_\xi \partial^a_\tau m^{++}_j(\xi, \tau)| \lesssim |\tau|^{1/2 - a} \|(\xi, \tau)\|^{-(|\gamma| + 1)},
	\end{equation}
	for $0 \leq a + |\gamma| \leq n + 4$ and that the support of $m^{++}_j$ is contained in
	\begin{equation}
		\label{E:m++:supp}
		\left\lbrace (\xi, \tau) \st 0 \leq |\tau| \leq \|(\xi, \tau)\|^2/2 \right\rbrace .
	\end{equation}
	Using these $|m^{++}_j(\xi, \tau)| \lesssim 1$ and by the same argument as before $T_{m_j^{++}}$ exists and is bounded on $L^p$ for $1 < p < \infty$.
\end{proof}

So far we have only studied this parabolic Sobolev space $L^p_{1,1/2}$ on $\R^n$ however our aim is to work on the boundary $\partial\Omega$ where $\Omega$ is as in \cref{D:domain}.
\begin{definition}[Parabolic Sobolev spaces on $\Lip(1,1/2)$ cylinders]
	\label{D:paraSob:domain}
	Let $1 < p < \infty$ and $\Omega$ be a $\Lip(1,1/2)$ cylinder as in \cref{D:domain} with local mappings $\phi_j: U \to 8\Z_j\cap\partial\Omega$, where $U$ is the upper half space.
	Let $\eta_j$ be a smooth partition of unity of $\partial\Omega$ with the following properties:
	\begin{enumerate}
		\item $0 \leq \eta_j \leq 1$,
		\item $\sum \eta_j = 1$,
		\item the $\eta_j$ have bounded overlap: i.e.\ for each fixed $(x,t)$ $\#\{j: \eta_i(x,t) > 0\} \leq M$ and
		\item $\supp \eta_j \subset B_{r_j}(x_j,t_j)$ with $r_j \sim \sup_\tau \diam(\Omega_\tau)$.
	\end{enumerate}
	We then define the $L^p_{1,1/2}$ norm on $\partial\Omega$ as
	\begin{equation}
		\|f\|_{L^p_{1,1/2}(\partial\Omega)} =
		\left( \sum_j \left( \|\D \left( (f\eta_j) \circ\phi_j\right) \|^p_{L^p(\R^n)}
		+ \|(f\eta_j) \circ\phi_j\|^p_{L^p(\R^n)} \right)\right)^{1/p}.
	\end{equation}
	By the relationship in \cref{T:Dn} this is equivalent to
	\begin{equation}
		\begin{split}
			\label{E:paraSob:domain}
			\|f\|^p_{L^p_{1,1/2}(\partial\Omega)} \sim
			\sum_j \Big( &\|\nabla\left( (f\eta_j) \circ\phi_j\right) \|^p_{L^p(\R^n)} + \|D^t_{1/2}\left( (f\eta_j) \circ\phi_j\right) \|^p_{L^p(\R^n)} \\
			&+ \|(f\eta_j) \circ\phi_j\|^p_{L^p(\R^n)} \Big).
		\end{split}
	\end{equation}
	It can be shown that when $\partial\Omega = \R^n$ the norm defined here is equivalent to the one given in \cref{D:paraSob}.
\end{definition}

\subsection{$L^p$ Regularity and $L^p$ Dirichlet Boundary Value Problems}

We are now in the position to define the $L^p$ regularity and $L^p$ Dirichlet problems.

\begin{definition}[\cite{Aro68}]
	\label{D:weak}
	We say that $u$ is a \textit{weak solution} to a parabolic operator of the form \cref{E:pde} in $\Omega$ if $u, \nabla u \in L^2_\loc(\Omega)$, $\sup_t \|u(\cdot, t)\|_{L^2_\loc (\Omega_t)} < \infty$ and
	\begin{equation*}
		\int_\Omega (-u \phi_t + A\nabla u \cdot \nabla \phi) \dX\dt = 0
	\end{equation*}
	for all $\phi \in C^\infty_0(\Omega)$.
	A weak solution to the adjoint operator \cref{E:pde:adjoint} is defined similarly.
\end{definition}

\begin{definition}
	\label{D:Rp}
	In light of \cref{S:paraSobolev}, following \cites{Bro89,Bro90,HL96,HL99} we say the $L^p$ \textit{Regularity problem} for the equation \cref{E:pde} is solvable if the the unique solution $u$ of this equation in $\Omega$ with boundary data $f\in C(\partial\Omega)\cap L^p_{1,1/2}(\partial\Omega, \dsig)$
	satisfies the following non-tangential maximal function estimate
	\begin{equation}
		\label{E:Rp}
		\|\widetilde{N}_2(\nabla u)\|_{L^p(\partial\Omega, \dsig)} \lesssim \|f\|_{L^p_{1,1/2}(\partial\Omega, \dsig)},
	\end{equation}
	with the implied constants depending only on the ellipticity constants, $n, p$ and triple $(\ell, N,C_0)$ of \cref{D:domain}.
	Here $\widetilde{N}_2$ denotes the $L^2$ based nontangential maximal function. When \eqref{E:Rp} holds we say that the equation  \cref{E:pde} has the property $(R)_p$ in $\Omega$.
\end{definition}

Here the use of the $L^2$ based non-tangential maximum function is natural since $\nabla u\in L^2_{loc}(\Omega)$.
In general better smoothness of the gradient cannot be expected unless we assume more smoothness of the coefficients of the parabolic operator.

\begin{remark}
	Some authors \cites{Bro87,Mit01,Nys06,CRS15} also require \newline $\|\widetilde{N}_2(D^t_{1/2} u)\|_{L^p} \lesssim \|f\|_{L^p_{1,1/2}}$ or $\|\widetilde{N}_2(HD^t_{1/2} u)\|_{L^p} \lesssim \|f\|_{L^p_{1,1/2}}$, where $H$ is the Hilbert transform in the time variable. For our result we do not assume this, hence our notion of solvability is slightly weaker than that of the authors above. It follows therefore that the $(R)_p$ solvability in the sense of
	\cites{Bro87,Mit01,Nys06,CRS15} implies solvability in the sense of \cref{D:weak}.
\end{remark}

\begin{definition}
	\label{D:Dp}
	We say the $L^p$ \textit{Dirichlet problem} for the equation  \cref{E:pde} is solvable if the the unique solution $u$ of this equation in $\Omega$ with boundary data $f\in C(\partial\Omega)\cap L^p(\partial\Omega, \dsig)$
	satisfies the following non-tangential maximal function estimate
	\begin{equation}
		\label{E:Dp}
		\|N(u)\|_{L^p(\partial\Omega, \dsig)} \lesssim \|f\|_{L^p(\partial\Omega, \dsig)},
	\end{equation}
	with the implied constant depending only on the ellipticity constants, $n, p$ and triple $(\ell, N,C_0)$ of \cref{D:domain}. When \eqref{E:Dp} holds we say that the equation \cref{E:pde} has the property $(D)_p$ in $\Omega$.
	The property $(D^*)_{p'}$ for the adjoint equation \cref{E:pde:adjoint} is defined analogously and is equivalent to
	solvability of the $L^{p'}$ Dirichlet problem for the equation \eqref{E:pde:adjoint:reflected} in the domain $\tilde\Omega$.
\end{definition}

\begin{remark}
	\label{R:Dp}
	It is well known that the $L^p$ solvability of the Dirichlet problem for some $1< p <\infty$ is equivalent to the parabolic measure $\omega$ belonging to a ``parabolic $A_\infty$" class with respect to the measure $\sigma$ on the surface $\partial\Omega$, \cite{Nys97}*{Theorem 6.2}.
	More specifically, the property $(D)_{p'}$ is equivalent to $\omega \in B_p(\mathrm{d}\sigma)$.
\end{remark}

We now recall the definition of parabolic $A_\infty$ and $B_p$.

\begin{definition}[$A_\infty$ and $B_p$]
	\label{D:Bp}
	Let $\Omega$ be a $\Lip(1,1/2)$ cylinder from \cref{D:domain}.
	For a ball $\Delta_d$ with radius $d\lesssim \sup_{\tau}\diam(\Omega_\tau)$ we denote its corkscrew point by $V_d$.
	We say that the parabolic measure $\omega^{V_d}$ of \cref{E:pde} is in $A_\infty(\Delta_d)$ if for every $\epsilon>0$ there exists $\delta = \delta(\epsilon) > 0$ such that for any ball $\Delta \subset \Delta_d$ and subset $E \subset \Delta$ we have
	\begin{equation}
		\frac{\omega^{V_d}(E)}{\omega^{V_d}(\Delta)} < \delta \Longrightarrow \frac{\sigma(E)}{\sigma(\Delta)} < \epsilon.
	\end{equation}
	The measure $\omega$ is in $A_\infty$ if $\omega^{V_d}$ belongs to $A_\infty(\Delta_d)$ for all $\Delta_d$.
	If $A_\infty$ holds then $\omega^{V_d}$ and $\sigma$ are mutually absolutely continuous and hence one can write $\dw^{V_d}=K^{V_d} \dsig$.

	For $p\in (1,\infty)$ we say that $\omega$ belongs to the \textit{reverse-H\"older class} $B_p(\mathrm{d}\sigma)$ if for all $\Delta_d$ the kernel $K^{V_d}$ satisfies the reverse H\"older inequality
	\begin{equation}
		\label{E:Bp}
		\left(\sigma(\Delta)^{-1}\int_{\Delta} \left(K^{V_d}\right)^p \dsig \right)^{1/p}
		\lesssim \sigma(\Delta)^{-1}\int_{\Delta} K^{V_d} \dsig,
	\end{equation}
	for all balls $\Delta\subset \Delta_d$.
\end{definition}

\begin{remark}
	$A_\infty=\bigcup_{p>1} B_p$.
\end{remark}

\section{Basic Results and Interior Estimates}
\begin{lemma}[Poincar\'{e} type inequality, \cite{Zie89}*{Cor.~4.5.3}]
	\label{L:Poincare}
	If $u \in W^{1,p}(E)$ and $p>1$ then
	\begin{equation}
		\|u\|_{L^{p^*}(E)} \leq C(B_{1,p}(N))^{-1/p} \|Du\|_{L^{p}(E)},
	\end{equation}
	where $B_{\alpha,p}(E)$ is the Bessel capacity of the set $E$
	\footnote{See \cite{Zie89} for a definition of Bessel capacity.},
	$N$ is the set where $u$ vanishes, i.e.\ $N = \{x \st u(x) = 0\}$,
	and $p^* = \frac{np}{n-p}$ if $p < n$, $1 \leq p^* < \infty$ if $p=n$ and $p^* = \infty$ if $p > n$.
\end{lemma}

In our work we use this for the case where $E$ is a time slice of $T(\Delta_r)$.
\begin{corollary}
	\label{cor:Poincare}
	Let $u \in W^{1,p}(\left.T(\Delta_r)\right| _{t'})$, where $u = 0$ on $\Delta_r|_{t'}$ for some fixed time $t'$.
	Let $p>1$ then there is a constant $C$ independent of $r$ such that
	\begin{equation}
		\|u\|_{L_x^{p}\left(\left.T(\Delta_r)\right| _{t'}\right)} \leq Cr \|\nabla u\|_{L_x^p\left(\left.T(\Delta_r)\right| _{t'}\right)}.
	\end{equation}
\end{corollary}
\begin{proof}
	The case for $r = 1$ follows from the positivity of $B_{\alpha,p}\left( \left.\Delta_1 \right| _{t'}\right) $ \cite{Zie89}*{\S 2.6}, \cref{L:Poincare}, and H\"{o}lder's inequality.
	For a general $r$ apply the substitution $v(x) := u(rx)$ then $v \in W^{1,p}\left( \left.T(\Delta_r)\right| _{t'}\right) $ and applying the $r=1$ case and a change of variables gives the general result.
\end{proof}

We now recall some foundational estimates needed to prove the main theorem.

\begin{lemma}[A Cacciopoli inequality, see \cite{Aro68}]
	\label{L:Caccio}
	Let $A$ satisfy \cref{E:elliptic} and suppose that $u$ is a weak solution of \cref{E:pde} or \cref{E:pde:adjoint} in $Q_{4r}(X,t)$ with $0 < r < \delta(X,t)/8$.
	Then there exists a constant $C=C(\lambda,\Lambda, n)$ such that
	\begin{equation*}
		\begin{split}
			r^{n} \left(\sup_{Q_{r/2}(X,t)} u \right)^{2}
			&\leq C \sup_{t-r^2 \leq s \leq t+r^2} \int_{Q_r(X)} u^{2}(Y,s) \dY
			+ C\int_{Q_{r}(X, t)} |\nabla u|^{2} \dY\ds \\
			&\leq \frac{C^2}{r^2} \int_{Q_{2r}(X, t)} u^{2}(Y, s) \dY\ds.
		\end{split}
	\end{equation*}
\end{lemma}

Lemmas 3.4 and 3.5 in \cite{HL01} give us the following estimates for weak solutions of \cref{E:pde} or \cref{E:pde:adjoint}.

\begin{lemma}[Interior H\"{o}lder continuity]
	\label{L:int_Holder}
	Let $A$ satisfy \cref{E:elliptic} and suppose that $u$ is a weak solution of \cref{E:pde} or \cref{E:pde:adjoint} in $Q_{4r}(X,t)$ with $0 < r < \delta(X,t)/8$.
	Then for any $(Y,s), (Z,\tau) \in Q_{2r}(X,t)$
	\[
		\left|u(Y, s) - u(Z, \tau)\right| \leq C \left( \frac{\|(Y,s) - (Z,\tau)\|}{r}\right)^{\alpha} \sup_{Q_{4r}(X,t)} |u|,
	\]
	where $C=C(\lambda, \Lambda, n)$, $\alpha=\alpha(\lambda,\Lambda,n)$, and $0 < \alpha < 1$.
\end{lemma}

\begin{lemma}[Harnack inequality]
	\label{L:Harnack}
	Let $A$ satisfy \cref{E:elliptic} and suppose that $u$ is a weak non-negative solution of \cref{E:pde} in $Q_{4r}(X,t)$, with $0 < r < \delta(X,t)/8$.
	Suppose that $(Y,s), (Z,\tau) \in Q_{2r}(X,t)$ then there exists $C=C(\lambda, \Lambda, n)$ such that, for $\tau < s$,
	\[
		u(Z, \tau) \leq u(Y, s) \exp \left[ C\left( \frac{|Y-Z|^2}{|s-\tau|} + 1 \right) \right].
	\]
	If $u \geq 0 $ is a weak solution of \cref{E:pde:adjoint} then this inequality holds when $\tau > s$.
\end{lemma}

We state a version of the maximum principle from \cite{DH16} that is a modification of Lemma 3.38 from \cite{HL01}.
\begin{lemma}[Maximum Principle]
	\label{L:MP}
	Let $A$ satisfy \cref{E:elliptic}, $\Omega$ be a $\Lip(1,1/2)$ cylinder and let $u$, $v$ be bounded continuous weak solutions to \cref{E:pde} in $\Omega$.
	If $|u|,|v|\to 0$ uniformly as $t \to -\infty$ and
	\[
		\limsup_{(Y,s)\to (X,t)} (u-v)(Y,s) \leq 0
	\]
	for all  $(X,t)\in\partial\Omega$, then $u\leq v$ in $\Omega$.
\end{lemma}

\begin{remark}[\cite{DH16}]
	\label{R:max}
	The proof of Lemma 3.38 from \cite{HL01} works given the assumption that $|u|,|v|\to 0$ uniformly as $t\to-\infty$.
	Even with this additional assumption,  the lemma as stated is sufficient for our purposes.
	We shall mostly use it when $u\le v$ on the boundary of $\Omega\cap\{t\ge \tau\}$ for a given time $\tau$.
	Obviously then the assumption that $|u|,|v|\to 0$ uniformly as $t\to-\infty$ is not necessary.
	Another case when the Lemma as stated here applies is when $u|_{\partial\Omega},v|_{\partial\Omega}\in C_0(\partial\Omega)$, where $C_0(\partial\Omega)$ denotes the class of continuous functions decaying to zero as $t\to\pm\infty$.
	This class is dense in any $L^p(\partial\Omega,d\sigma)$, $p<\infty$ allowing us to consider an extension of the solution operator from $C_0(\partial\Omega)$ to $L^p$.
\end{remark}

The following Carleson type estimate was proved for Lipschitz cylinders in \cite{Sal81} and extended to $\Lip(1,1/2)$ cylinders in \cite{Nys97}*{Lemma 2.4}.

\begin{lemma}[Carleson type estimate, \cite{Nys97}]
	\label{L:carl}
	Let $\Omega$ be a $\Lip(1,1/2)$ cylinder from \cref{D:domain} with character $(\ell,N,C_0)$ and $A$ satisfy \cref{E:elliptic}.
	Let	$u$ be a non-negative weak solution of \cref{E:pde} or the adjoint \cref{E:pde:adjoint} in $\Psi_{2r}(y,s)$ for $(y,s) \in \partial\Omega$ and $0 < r < r_0/2$.
	Let $u$ vanish continuously on $\Psi_{2r}(y,s)\cap \partial\Omega$, then there exists $C = C(\ell, \lambda, \Lambda, n)$ such that for $(X,t) \in \Psi_r(y,s)$
	\begin{equation}
		\label{E:carl}
		u(X,t) \leq C u(V^\pm_r),
	\end{equation}
	where the plus sign is taken when $u$ is a weak solution of \cref{E:pde} and the minus sign is taken when $u$ is a weak solution of the adjoint \cref{E:pde:adjoint}. Here $V^+_r$ is the usual (forward in time) corkscrew point of $\Delta_r(y,s)$, while $V^-_r$ is backward-time corkscrew point $\Delta_r(y,s)$ (i.e.\ a point at time $s-2r^2$).
\end{lemma}

\begin{lemma}[Parabolic doubling, corkscrew point, see \cite{Nys97} for more general statements in time-varying domains]
	\label{L:doubling}
	Let $\Omega$ be a $\Lip(1,1/2)$ cylinder from \cref{D:domain} with character $(\ell,N,C_0)$.
	Let $\Delta_{2 r} \subset \Delta_d$ be boundary balls, and $V_{2r}$ and $V_d$ be their corkscrew points. Let $A$ satisfy \cref{E:elliptic} and $\omega^{V_d}$ be the parabolic measure of \cref{E:pde}.
	Then there exists $C = C(\lambda, \Lambda, n, \ell)$ such that
	\begin{enumerate}
		\item $\omega^{V_d} (\Delta_d) \geq C$

		\item $\omega^{V_d} (\Delta_{2 r}) \leq C \omega^{V_d}(\Delta_r) \quad$ (doubling)

		\item If $E \subset \Delta_{2 r}$ is a Borel set then
		      \[
			      \omega^{V_{2r}} (E) \sim  \frac{\omega^{V_d}(E)}{\omega^{V_d} (\Delta_{2 r})}.
		      \]
	\end{enumerate}
\end{lemma}

The next lemma shows that the parabolic measure of different corkscrew points of large balls are comparable.
\begin{lemma}[Change of corkscrew point]
	\label{L:changeCork}
	Let $\Omega$ be a $Lip(1,1/2)$ cylinder.
	Let $\Delta_r(y,s)$ be a boundary ball with $r \sim \sup_\tau \diam\Omega_\tau$ and $V_r$ and $V_r'$ be two corkscrew points of $\Delta_r(y,s)$ both later in time than $s + (2r)^2$.
	Let	$\omega^{V_r}$ be the parabolic measure of \cref{E:pde}, $A$ satisfy \cref{E:elliptic} and $E \subset \Delta_r(y,s)$ be a Borel set then
	\begin{equation}
		\label{E:changeCork}
		\omega^{V_r}(E) \sim \omega^{V_r'}(E).
	\end{equation}
	The same result holds with the adjoint parabolic measure $\omega^{*V_r}$, and $V_r$ and $V_r'$ are corkscrew points earlier in time than $s - (2r)^2$.
\end{lemma}

\begin{proof}
	The idea of this proof is to view $\omega^{V_r}(E)$ as $u(V_r)$, where $u$ is the solution of \cref{E:pde} with boundary data $\chi_{E}$ and $\chi$ is the usual indicator function.
	We then set up to apply the maximum principle to an appropriately chosen domain $\partial\Omega \cap \{t \geq s'\}$.

	Let $\Delta_{r/2}(y',s')$ be a boundary ball later in time than $\Delta_r(y,s)$ so that $E$ and $\Delta_{r/2}(y',s')$ are disjoint.
	Therefore the boundary data is $0$ there and we can apply \cref{L:carl} to control $u$ in $\Psi_{r/4}(y',s')$ by $u(V^+)$, where $V^+$ is a corkscrew point of $\Delta_{r/2}(y',s')$ and at a time earlier than $(2r)^2$.

	Since $r \sim \diam\Omega_{s'}$ using Harnack chains, the Harnack inequality (\cref{L:Harnack}) and by varying $y'$ we can uniformly control $u$ at the time $s'$ by $u(V_r)$, that is we have $u(X,s')\lesssim u(V_r)$ for all $(X,s')\in\Omega_{s'}$.
	It follows by the maximum principle in \cref{R:max} applied to the domain to $\partial\Omega \cap \{t \geq s' \}$ that $u(X,t)\lesssim u(V_r)$ for all $(X,t)\in\Omega\cap \{t \geq s' \}$.
	In particular, $u(V'_r) \lesssim u(V_r)$ and therefore $\omega^{V_r}(E) \lesssim \omega^{V_r'}(E)$. Exchanging the roles of $V_r$ and $V_r'$ gives the other inequality.
\end{proof}

We use the following properties of the Green's function. The existence of the Green's functions $G$ and $G^*$ in $\Omega$ for \eqref{E:pde}, \eqref{E:pde:adjoint}, respectively is well known and follows from H\"{o}lder continuity and a Perron-Wiener-Brelot style argument.

\begin{lemma}[{\cite{Fri64}}]
	\label{L:Green}
	Let $\Omega$ be a $Lip(1,1/2)$ cylinder and $A$ satisfy \cref{E:elliptic} then the Green's function $G$ for \cref{E:pde} has the following properties.
	\begin{enumerate}
		\item $G(X,t,Y,s) = 0$ for $s > t$, $(X,t)$, $(Y,s) \in \Omega$.

		\item For fixed $(Y,s) \in \Omega$, $G(\cdot, Y,s)$ is a solution to \cref{E:pde} in $U \setminus \{(Y,s)\}$.

		\item For fixed $(X,t) \in \Omega$, $G(X,t, \cdot)$ is a solution to \cref{E:pde:adjoint}, the adjoint equation in $\Omega \setminus \{(X,t)\}$.

		\item If $(X,t)$, $(Y,s) \in \Omega$ then $G(X,t, \cdot)$ and $G(\cdot, Y,s)$ extend continuously to $\overline{\Omega}$ provided both functions are defined to be zero on $\partial \Omega$.
	\end{enumerate}
\end{lemma}

The following lemma is a consequence of \cite{Nys97}. We state it for the adjoint equation \eqref{E:pde:adjoint} in $\Omega$ as we apply the lemma in this context.
This lemma was originally stated in Lipschitz cylinders in \citelist{\cite{FGS86}*{Theorem 1.4} \cite{FS97}*{Theorem 4}} and was extended to the domains in question by \cite{Nys97}.
\begin{lemma}
	\label{L:Green:comparison}
	Let $\Omega$ be a $Lip(1,1/2)$ cylinder, $A$ satisfy \cref{E:elliptic}, $G^*$ be Green's function and $\omega^*$ be the parabolic measure associated to \cref{E:pde:adjoint}.
	Let $\Delta_r\subset\Delta_d$ be the surface balls on $\partial\Omega$ such that $\Delta_{2r}\subset\Delta_d$ and $d\lesssim \frac{r_0}{C_0}$. Then there exists constants depending on $n$, $\lambda$ and $\Lambda$ and character of the domain $\Omega$ such that

	\begin{equation}
		r^n G^*( V^-(\Delta_d),V^-(\Delta_r)) \sim \omega^{*V^-(\Delta_d)}(\Delta_r).
	\end{equation}
	Here $V^-(\Delta_r)$ and $V^-(\Delta_d)$ are backward in time corkscrew points as in Lemma \ref{L:carl}.
\end{lemma}

\section{Proof of \Cref{T:Rp}}
\label{sec:RimpliesD_proof}
This proof uses some of the ideas from Kenig and Pipher's \cite{KP93} proof in the elliptic setting.
However due to the time irreversibility of parabolic equations we do not have the comparison principle, the Carleson estimate \cite{CFMS81}*{Theorem 1.1} or Harnack's principle that they used.
Also the non-commutativity of taking the adjoint and the pullback mapping introduce additional difficulties.
Instead, we get around these problems using lemmas developed in \cite{Nys97}, the maximum principle, a different Carleson type estimate,  approaching some estimates from an integral instead of a pointwise point of view and using the Hardy-Littlewood maximal function.

Assume that $(R)_p$ holds for \cref{E:pde} and let $\omega^*$ be the parabolic measure associated to the adjoint equation \cref{E:pde:adjoint}.
By \cref{R:Dp} to show that $(D^*)_{p'}$ holds we need to show that $\omega^* \ll \sigma$, where $\sigma$ is the measure on $\partial\Omega$ in \cref{D:measure}, and $\omega^*$ belongs to the reverse H\"{o}lder class $B_p(\mathrm{d}\sigma)$, see \cref{D:Bp}.

\begin{steps}
	\item Preliminaries

	We first prove \cref{E:Bp} for surface balls that fit inside a cylinder $2\Z_j$ and then use a covering argument to show that \cref{E:Bp} holds for all balls with the correct scaling.
	Note that since \cref{E:comp} holds in $2\Z_j$ so we can replace $\sigma$ by $\H^n$.

	Let $\Delta_d$ be a surface ball on $\partial \Omega$ with $d \lesssim \frac{r_0}{C_0}$ then $\Delta_d$ lies completely inside an $\ell$-cylinder $2\Z_j$.
	After we apply $\phi_j$, the pullback transformation, $\Delta_d$ becomes a surface ball $\Delta_d$ on $\partial U$, where $U$ is the upper half space.
	Let $\Delta_r(y,a) \subset \Delta_d \subset U$ be a surface ball such that $4r < d$.
	Note if we omit the point that $\Delta_r$ is centred at then it will be centred at $(y,a) \in \partial U$.

	As in \cite{KP93}, we define a non-negative $C^\infty$ function $f$ on $\partial U$ as follows:
	$f = 0$ on $\Delta_r$, $f = 1$ on $\Delta_{3r}\backslash\Delta_{2r}$ and $f=0$ on $\partial U \backslash \Delta_{4r}$
	with $|\nabla_T f| \lesssim \frac{1}{r}$ and $|\partial_t f| \lesssim \frac{1}{r^2}$.
	Here we note that $\Delta_{4r}\subset\Delta_d$.
	Using \cref{T:Dn} and interpolation we have
	\begin{equation}
		\label{E:step1:grad-f}
		\begin{split}
			\int_{\partial U} |\nabla_T f|^p \dH^n &\lesssim r^{n+1-p}, \\
			\int_{\partial U} |\D_n f|^p \dH^n &\lesssim \|D^t_{1/2}f\|^p_{L^p} + \|\nabla f\|^p_{L^p} \lesssim \|\partial_t f\|^{p/2}_{L^p} \|f\|^{p/2}_{L^p} + \|\nabla f\|^p_{L^p} \lesssim r^{n+1-p}.
		\end{split}
	\end{equation}
	By Sobolev embedding, since $f \in C^\infty_0(\Delta_d)$, for a fixed time $t$,
	$\int_{\R^{n-1}} |f(x,t)|^p \dx \lesssim \int_{\R^{n-1}} |\nabla_T f(x,t)|^p \dx.$ Here and in the following estimate the implied constant will depend on $d$. Integrating the previous estimate in time gives
	\begin{equation}
		\label{E:step1:f}
		\int_{\partial U} |f|^p \dH^n \lesssim \int_{\partial U} |\nabla_T f|^p \dH^n \lesssim r^{n+1-p}.
	\end{equation}

	It follows that $f^u=f\circ \phi_j^{-1}$ is $\Delta_d$ supported boundary data on $\partial\Omega$ with $L^p_{1,1/2}(\partial\Omega,d\sigma)$ norm comparable to $r^{(n+1)/p-1}$.

	Since we assume $(R)_p$ solvability for the equation \eqref{E:pde} let $u$ be the solution  of \cref{E:pde} in $\Omega$ with boundary data $f^u$. It follows that we have for $u$ the following estimate
	\begin{equation}\label{NF}
		\|\tilde{N}(\nabla u)\|_{L^p(\partial\Omega)}\lesssim r^{(n+1)/p-1}.
	\end{equation}

	Let $s \ll r$ (we are going to take limit $s\to 0+$) and let $P\in\partial\Omega$ be a point on the boundary such that $\Delta_{10s}(P) \subset \Delta_r$.

	\item Equivalence between the Green's function and the parabolic measure.

	We now have three surface balls $\Delta_s\subset\Delta_r\subset\Delta_d$. Let $V^-_s$, $V^-_r$ and $V^-_d$ be their corkscrew points shifted backwards in time from their centres by $100s^2$, $100r^2$ and $100d^2$ respectively.
	Therefore, by applying \cref{L:Green:comparison} we have
	\begin{equation}
		\label{E:comparison}
		\frac{\omega^{*V^-_d}(\Delta_s(P))}{\omega^{*V^-_d}(\Delta_r)}  \sim \frac{s^n}{r^n} \frac{G^*(V^-_d,V^-_s)}{G^*(V^-_d,V^-_r)}=\frac{s^n}{r^n} \frac{G(V^-_s,V^-_d)}{G(V^-_r,V^-_d)}.
	\end{equation}

	\item Controlling Green's function by the solution $u$.
	\label{T:step3}

	For the next step in this proof we want to show that \cref{E:comparison} can be uniformly controlled by $u(V^-_s) s^n/r^n$ for all $s \ll r$.
	To this end, we show that $G(X,t, V^-_d) \lesssim u(X,t) G(V^-_r, V^-_d)$ on the boundary of $T(\Delta_{5r/2})$ and then apply the maximum principle, \cref{L:MP}, to show that $G(X,t, V^-_d) \lesssim u(X,t) G(V^-_r, V^-_d)$ for $(X,t)\in T(\Delta_{5r/2})$.

	On $\Delta_{5r/2}$ we have that $0 = G(\cdot,V^-_d) \leq u(\cdot) G(V^-_r, V^-_d)$ so we are left to show that $G^{V^-_d}(X) \lesssim G^{V^-_d}(V^-_r)$ and $u \sim 1$ on $\partial T(\Delta_{5r/2}) \backslash \partial \Omega$.

	\begin{steps}
		\item $G(X,t, V^-_d) \lesssim G(V^-_r, V^-_d)$ on $\partial T(\Delta_{5r/2}) \backslash \partial \Omega$.
		Here we use that $T(\Delta_{5r/2})$ is later than $V^-_r$ in time, i.e. $T(\Delta_{5r/2}) \subset \{(X,t) : t > a - (9r)^2\}$. For points $(X,t)$ in $\partial T(\Delta_{5r/2})$ away from $\partial \Omega$ we can just apply the interior Harnack inequality to conclude that $G(X,t, V^-_d) \lesssim G(V^-_r, V^-_d)$.	For points $(X,t)$ near $\partial \Omega$ we can apply \cref{L:carl}, to obtain $G(X,t,V^-_d) \lesssim G(V^-(\Delta_{r}(z,\tau)), V^-_d)$, where $(z,\tau)$ is any point in $\Delta_{5r/2}$.
		Since $V^-_r$ is at an earlier time than $V^-(\Delta_{r}(z,\tau))$, we can again apply the Harnack inequality, \cref{L:Harnack}, to obtain $G(X,t,V^-_d) \lesssim G(V^-_r, V^-_d)$ for $(X,t) \in T(\Delta_{r}(z,\tau))$. From this the claim follows.

		\item $u \sim 1$ on $\partial T(\Delta_{5r/2}) \backslash \partial \Omega$

		As before, near to $\partial \Omega$ applying \cref{L:carl} to $1-u$ gives us that $u(X,t) \sim 1$ for $(X,t) \in \Psi_{r/4}(z,\tau)$, where $(z,\tau)\in\partial \Delta_{5r/2}$.
		Away from $\partial \Omega$ we use interior Harnack's inequality to conclude that $u \sim 1$ at a later time when $\partial T(\Delta_{5r/2}) \cap \partial \Omega$.

		\item Applying the maximum principle

		Therefore, by applying the maximum principle, we have that $G(X,t, V^-_d) \lesssim u(X,t) G(V^-_r, V^-_d)$ for $(X,t) \in T(\Delta_{5r/2})$ and since $V^-_s\in T(\Delta_{5r/2})$ conclude that $G(V^-_s, V^-_d) \lesssim u(V^-_s) G(V^-_r, V^-_d)$.

		We have now proved
		\begin{equation}
			\frac{\omega^{*V^-_d}(\Delta_s(P))}{\omega^{*V^-_d}(\Delta_r)} \lesssim \frac{s^n}{r^n} u(V^-_s).
		\end{equation}
	\end{steps}

	\item

	Applying the Poincar\'{e} type inequality to the spacial variables, \cref{cor:Poincare}, for a fixed time $t=t'$ we have for $q > 1$
	\begin{equation*}
		\left(\fint_{\left.T(\Delta_s(P))\right| _{t'}} |u(X,t)|^q \dX \right)^{1/q} \lesssim s \left(\fint_{\left.T(\Delta_s(P))\right| _{t'}} |\nabla u(X,t)|^q \dX \right)^{1/q}.
	\end{equation*}
	Then averaging in time over $(a'-s^2, a'+s^2)$ gives
	\begin{equation}
		\label{E:Poincare:applied}
		\left(\fint_{T(\Delta_s(P))} |u(X,t)|^q \dX\dt \right)^{1/q}  \lesssim s \left(\fint_{T(\Delta_s(P))} |\nabla u(x,t)|^q \dX\dt \right)^{1/q}.
	\end{equation}
	By applying the Harnack inequality to $u(V^-_s)$,  we can estimate the value of $u$ at this point by
	the infimum of $u$ over the ball $Q_{s/8}\left(V^-_s + (0, s^2/4^2)\right)$ (the centre of this ball is $V^-_s$ shifted by $s/16$ in time).  It follows that

	\begin{equation*}
		\begin{split}
			u(V^-_s)
			&\lesssim \inf_{Q_{s/8}\left(V_s + (0, s^2/4^2)\right)} u  \lesssim \left(\fint_{Q_{s/2}(V_s)} |u(X,t)|^q \dX\dt \right)^{1/q} \\ &\lesssim s \left(\fint_{T(\Delta_{12s}(P))} |\nabla u(X,t)|^q \dX\dt \right)^{1/q}.
		\end{split}
	\end{equation*}
	Therefore
	\begin{equation}
		\label{E:endStep4}
		\frac{\omega^{*V^-_d}(\Delta_s(P))}{\omega^{*V^-_d}(\Delta_r)} \lesssim \frac{s^n}{r^n} u(V^-_s) \lesssim \frac{s^{n+1}}{r^n} \left(\fint_{T(\Delta_{12s}(P))} |\nabla u(X,t)|^q \dX\dt \right)^{1/q}.
	\end{equation}

	\item
	We would like to bound this by $\widetilde{N}_2(\nabla u)(P)$, the $L^2$ based non-tangential maximal function. This is easy to do in the elliptic setting but it is not clear whether it is possible to do in our setting due to the time irreversibility of the parabolic PDE. Instead we clearly have the following bound
	\begin{equation*}
		\left(\fint_{T(\Delta_{12s}(P))} |\nabla u(X,t)|^q \dX\dt \right)^{1/q} \lesssim \left(M\left(\widetilde{N}_q(\nabla u)^q\right)(P) \right)^{1/q},
	\end{equation*}
	where $M$ is the parabolic version of the Hardy-Littlewood maximal function defined using parabolic boundary balls.

	Combining this estimate with \cref{E:endStep4} we have
	\begin{equation}
		\label{E:endStep5}
		\frac{\omega^{*V^-_d}(\Delta_s(P))}{\sigma(\Delta_s(P))} \lesssim \frac{\omega^{*V^-_d}(\Delta_r)}{r^n} \left(M\left(\widetilde{N}_q (\nabla u)^q\right)(P) \right)^{1/q},
	\end{equation}
	where as before $s < r/10$ and $P$ is such that $\Delta_{10s}(P) \subset \Delta_r$.
	In particular, this estimate holds for $P \in \Delta_{r/2}$.

	\item The $B_p$ condition.

	To show the property $(D^*)_{p'}$ we need to show that $K^{V^-_d} = \dw^{*V^-_d}/\dsig$ belongs to the reverse H\"{o}lder class $B_p(\mathrm{d}\sigma)$, c.f.\ \cref{E:Bp}.
	To do this we take the same approach as \cite{KP93}.
	Let
	\begin{equation*}
		h^{V^-_d}(P) := \sup_{s\in (0,r/10)} \frac{\omega^{*V^-_d}(\Delta_s(P))}{\sigma(\Delta_s(P))},
	\end{equation*}
	then $K^{V^-_d}(P) \leq h^{V^-_d}(P)$ for $P \in \Delta_{r/2}$.
	Since $(M(|f|^q))^{1/q}$ is $L^p$ bounded for $p > q>1$ and $\widetilde{N}_q (f) \leq \widetilde{N}_2 (f)$ for $0 < q \leq 2$ we choose $q\in (1,\min\{2,p\})$ to conclude
	\begin{equation}
		\begin{split}
			\label{E:k:Lp-bound}
			\|K^{V^-_d}\|_{L^p(\mathrm{d}\sigma)}
			&\leq \|h^{V^-_d}\|_{L^p(\mathrm{d}\sigma)}
			\lesssim \frac{\omega^{*V^-_d}(\Delta_r)}{r^n} \left\|\widetilde{N}_q (\nabla u)\right\|_{L^p(\mathrm{d}\sigma)} \\
			&\lesssim \frac{\omega^{*V^-_d}(\Delta_r)}{r^n} \left\|\widetilde{N}_2 (\nabla u)\right\|_{L^p(\mathrm{d}\sigma)}
			\lesssim \frac{\omega^{*V^-_d}(\Delta_r)}{r^n}\left\|f\right\|_{L^p_{1,1/2}(\mathrm{d}\sigma)}  < \infty.
		\end{split}
	\end{equation}
	Therefore, $K^{V^-_d},h^{V^-_d} \in L^p(\mathrm{d}\sigma)$ and so $\omega^{*V^-_d} \ll \sigma$.

	Using \cref{E:k:Lp-bound} and \cref{NF} the weight $K^{V^-_d}$ satisfies the $B_p$ condition \cref{E:Bp} for the ball $\Delta_{r/2}$
	\begin{equation}
		\begin{split}
			\label{E:step6:Bp}
			\left(\frac{1}{\sigma\left(\Delta_{r/2} \right)}\int_{\Delta_{r/2}} \left( K^{V^-_d} \right)^p \dsig \right)^{1/p}
			&\lesssim \frac{\omega^{*V^-_d}(\Delta_r)}{r^n} \left( \frac{1}{r^{n+1}} r^{n+1-p} \right)^{1/p} \\
			&\lesssim \frac{\omega^{*V^-_d}(\Delta_r)}{\sigma(\Delta_r)}
			\lesssim \frac{\omega^{*V^-_d}(\Delta_{r/2})}{\sigma(\Delta_{r/2})}.
		\end{split}
	\end{equation}
	By considering  different balls $\Delta_r$ we can conclude that the above inequality holds for any boundary ball $\Delta_r \subset \Delta_d$ with $4r \leq d \lesssim \tfrac{r_0}{C_0}$. One can then use Lemma \ref{L:changeCork} to see the above reverse H\"{o}lder inequality holds for all balls up to size $d$.
	It follows that the $L^{p'}$ Dirichlet problem for the adjoint PDE \eqref{E:pde:adjoint} is solvable in $\Omega$. \qed
\end{steps}


\end{document}